\documentclass{amsart}
\usepackage[utf8]{inputenc}
\usepackage[hidelinks]{hyperref}
\usepackage{amsthm, amsmath, amssymb, enumitem}
\usepackage{verbatim}

\newtheorem{corollary}{Corollary}
\newtheorem{lemma}{Lemma}
\newtheorem{proposition}{Proposition}
\newtheorem{theorem}{Theorem}

\theoremstyle{definition}
\newtheorem{definition}{Definition}
\newtheorem{example}{Example}

\theoremstyle{remark}
\newtheorem{remark}{Remark}

\DeclareMathOperator{\im}{im}
\DeclareMathOperator{\op}{\cdot_{op}}

\begin{document}

\title[Hilbert's basis theorem]{Hilbert's basis theorem for non-associative and hom-associative Ore extensions}
\author{Per B\"ack}
\address[Per B\"ack]{Division of Mathematics and Physics, The School of Education, Culture and Communication, M\"alar\-dalen  University,  Box  883,  SE-721  23  V\"aster\r{a}s, Sweden}
\email[corresponding author]{per.back@mdu.se}

\author{Johan Richter}
\address[Johan Richter]{Department of Mathematics and Natural Sciences, Blekinge Institute of Technology, SE-371 79 Karlskrona, Sweden}
\email{johan.richter@bth.se}


\begin{abstract}
We prove a hom-associative version of Hilbert's basis theorem, which includes as special cases both a non-associative version and the classical Hilbert's basis theorem for associative Ore extensions. Along the way, we develop hom-module theory. We conclude with some examples of both non-associative and hom-associative Ore extensions which are all Noetherian by our theorem.
\end{abstract}

\maketitle

\section{Introduction}
Hom-associative algebras are not necessarily associative algebras, the associativity condition being replaced by $\alpha(a)\cdot (b\cdot c)=(a\cdot b)\cdot \alpha(c)$, where $\alpha$ is a linear map referred to as a twisting map, and $a,b,c$ arbitrary elements in the algebra. Both associative algebras and non-associative algebras can thus be seen as hom-associative algebras; in the first case, by taking $\alpha$ equal to the identity map, and in the latter case by taking $\alpha$ equal to the zero map.

Historically hom-associative algebras originate in the development of hom-Lie algebras, the latter introduced by Hartwig, Larsson and Silvestrov as generalizations of Lie algebras, the Jacobi identity now twisted by a vector space homomorphism; the ``hom'' referring to this homomorphism~\cite{HLS03}. The introduction of these generalizations of Lie algebras was mainly motivated by an attempt to study so-called $q$-deformations of the Witt and Virasoro algebras within a common framework. Makhlouf and Silvestrov then introduced hom-associative algebras as the natural counterparts to associative algebras; taking a hom-associative algebra and defining the commutator as a new multiplication gives a hom-Lie algebra, just as with the classical relation between associative algebras and Lie algebras~\cite{MS06}. It was later discovered that there existed formally rigid associative algebras that could now be formally deformed when considered as hom-associative algebras~\cite{MS10a}, this indicating that hom-associative algebras could be useful in studying deformations as well. Since then, many papers have been written in the subject, and other algebraic structures have been discovered to have natural counterparts in the ``hom-world'' as well, such as e.g. hom-coalgebras, hom-bialgebras, and hom-Hopf algebras~\cite{MS10b,MS09}.

A class of algebras that contains many examples of formally rigid associative algebras such as e.g. the Weyl algebras and some universal enveloping algebras of Lie algebras, is that of Ore extensions. Ore extensions were first introduced under the name \emph{non-commutative polynomial rings} by Ore~\cite{Ore33}. Non-associative Ore extensions were introduced by Nystedt, {\"O}inert, and Richter in the unital case~\cite{NOR15} (see also \cite{NOR17} for a further extension to monoid Ore extensions). The construction was later generalized to non-unital, hom-associative Ore extensions by the authors of the present article and Silvestrov~\cite{BRS18}. They give examples, including hom-associative versions of the first Weyl algebra, the quantum plane, and a universal enveloping algebra of a Lie algebra, all of which are formal deformations of their associative counterparts~\cite{Bac19,BR19}.

In this paper, we prove a hom-associative version of Hilbert's basis theorem (\autoref{thm:hom-hilbert}), including as special cases both a non-associative version (\autoref{cor:non-hilbert}) and the classical associative Hilbert's basis theorem for Ore extensions (\autoref{re:assoc-hilbert}). In order to prove this, we develop hom-module theory and a notion of being hom-Noetherian (\autoref{sec:hom-module-theory}). Whereas the hom-module theory does not require a multiplicative idenity element, the hom-associative Ore extensions in this article are all assumed to be unital. We conclude with some examples of unital, non-associative and hom-associative Ore extensions which are  Noetherian as a consequence of our main theorem. In more detail, the article is organized as follows:

\autoref{sec:preliminaries} provides preliminaries from the theory of hom-associative algebras, and of unital, hom-associative Ore extensions as developed in~\cite{BRS18}.

\autoref{sec:hom-module-theory} deals with hom-modules over non-unital, hom-associative rings. Here, isomorphism theorems and the property of being hom-Noetherian are introduced. Whereas the notion of a hom-module was first introduced in \cite{MS10a}, the theory developed here is new as far as we can tell.

\autoref{sec:hilbert-theorem} contains the proof of a hom-associative version of Hilbert's basis theorem, including as special cases a non-associative and the classical associative version.

\autoref{sec:examples} contains examples of unital, non-associative and hom-associative Ore extensions which are all Noetherian by the aforementioned theorem.

\section{Preliminaries}\label{sec:preliminaries}
Throughout this paper, by \emph{non-associative} algebras we mean algebras which are not necessarily associative, including in particular associative algebras. We call a non-associative algebra $A$ \emph{unital} if there exists an element $1\in A$ such that for any element $a\in A$, $a\cdot 1=1\cdot a=a$. By \emph{non-unital} algebras, we mean algebras which are not necessarily unital, including unital algebras by definition.

\subsection{Hom-associative algebras}
This section is devoted to restating some basic definitions and general facts concerning hom-associative algebras.
\begin{definition}[Hom-associative algebra]\label{def:hom-assoc-algebra} A \emph{hom-associative algebra} over an associative, commutative, and unital ring $R$, is a triple $(M,\cdot,\alpha)$ consisting of an $R$-module $M$, a binary operation $\cdot\colon M\times M\to M$ linear over $R$ in both arguments, and an $R$-linear map $\alpha\colon M\to M$ satisfying, for all $a,b,c\in M$, $\alpha(a)\cdot(b\cdot c)=(a\cdot b)\cdot\alpha(c)$. The map $\alpha$ is referred to as the \emph{twisting map}.
\end{definition}

\begin{remark}
A hom-associative algebra over $R$ is in particular a non-unital, non-associative $R$-algebra, and in case $\alpha$ is the identity map, a non-unital, associative $R$-algebra. Moreover, any non-unital, non-associative $R$-algebra is a hom-associative $R$-algebra with twisting map equal to the zero map.
\end{remark}

\begin{remark}\label{re:unitality-condition}
If $A$ is a unital, hom-associative algebra, then $\alpha$ is completely determined by $\alpha(1)$ since $\alpha(a) =\alpha(a)(1\cdot 1)=a\cdot\alpha(1)$ for any $a\in A$.  
\end{remark}

\begin{definition}[Morphism of hom-associative algebras]\label{def:morphism} A \emph{morphism} from a hom-associative $R$-algebra $A:=(M,\cdot,\alpha)$ to a hom-associative $R$-algebra $A':=(M',\cdot',\alpha')$ is an $R$-linear map $f\colon M\to M'$ such that $f\circ \alpha= \alpha'\circ f$ and $f(a\cdot b)=f(a)\cdot'f(b)$ for all $a,b\in M$. If $f$ is bijective, the two are \emph{isomorphic}, written $A\cong A'$.
\end{definition}

\begin{definition}[Hom-associative subalgebra] Let $A:=(M,\cdot,\alpha)$ be a hom-associative algebra and $N$ a submodule of $M$ that is closed under the multiplication $\cdot$ and invariant under $\alpha$. The hom-associative algebra $(N,\cdot,\alpha\vert_N)$ is said to be a \emph{hom-subalgebra} of $A$.
\end{definition}

\begin{definition}[Hom-ideal] A \emph{right (left) hom-ideal} of a hom-associative $R$-algebra $A$ is an $R$-submodule $I$ of $A$ such that $\alpha(I)\subseteq I$, and for all $a\in A, i\in I$, $i\cdot a\in I$ ($a\cdot i\in I$). If $I$ is both a left and a right hom-ideal, it is simply called a \emph{hom-ideal}.
\end{definition}

Note that a hom-ideal of a hom-associative algebra $A$ is in particular a hom-subalgebra of $A$.

\begin{remark}\label{re:ideal}
In case a hom-associative algebra has twisting map equal to the identity map or the zero map, a right (left) hom-ideal is simply a right (left) ideal.
\end{remark}

\begin{definition}[Hom-associative ring]\label{def:hom-ring} A \emph{hom-associative ring} is a hom-associative algebra over the ring of integers.
\end{definition}

\begin{definition}[Opposite hom-associative ring] Let $S:=(R,\cdot,\alpha)$ be a hom-as\-so\-cia\-tive ring. The \emph{opposite hom-associative ring} of $S$, written $S^\text{op}$, is the hom-associative ring $(R,\cdot_\text{op},\alpha)$ where $r\cdot_\text{op} s:=s\cdot r$ for any $r,s\in R$.
\end{definition}

\subsection{Unital, non-associative Ore extensions}
In this section, we recall from \cite{BRS18} some basic definitions and results concerning unital, non-associative Ore extensions. We denote by $\mathbb{N}$ the set of all non-negative integers, and by $\mathbb{N}_{>0}$ the set of all positive integers. Let $R$ be a unital, non-associative ring, $\delta\colon R\to R$ and $\sigma\colon R\to R$  additive maps such that $\sigma(1)=1$ and $\delta(1)=0$. As a set, a $\emph{unital, non-associative Ore extension}$ of $R$, written $R[X;\sigma,\delta]$, consists of all formal sums $\sum_{i\in\mathbb{N}} a_i X^i$, called \emph{polynomials},  where only finitely many $a_i\in R$ are nonzero. We endow $R[X;\sigma,\delta]$ with the following addition and multiplication, holding for any $m,n\in\mathbb{N}$ and $a_i,b_i\in R$:
\begin{equation*}
\sum_{i\in\mathbb{N}}a_iX^i + \sum_{i\in\mathbb{N}}b_i X^i = \sum_{i\in\mathbb{N}}(a_i+b_i)X^i,\quad aX^m\cdot bX^n=\sum_{i\in \mathbb{N}}\left(a\cdot\pi^m_i(b)\right)X^{i+n}. 
\end{equation*}
Here, $\pi_i^m$, referred to as a \emph{$\pi$ function}, denotes the sum of all $\binom{m}{i}$ possible compositions of $i$ copies of $\sigma$ and $m-i$ copies of $\delta$ in arbitrary order. For instance, $\pi_1^2=\sigma\circ\delta+\delta\circ\sigma$. We also define $\pi_0^0:=\mathrm{id}_R$ and $\pi^m_i\equiv0$ whenever $i<0$, or $i>m$. The identity element $1$ in $R$ also becomes an identity element in $R[X;\sigma,\delta]$ upon identification with $1X^0$. We also think of $X$ as an element of $R[X;\sigma,\delta]$ by identifying it with the monomial $1X$. At last, defining two polynomials to be equal if and only if their corresponding coefficients are equal and imposing distributivity of the multiplication over addition make $R[X;\sigma,\delta]$ a unital, non-associative, and non-commutative ring. By identifying any $a\in R$ with $aX^0\in R[X;\sigma,\delta]$, we see that $R$ is a subring of $R[X;\sigma,\delta]$.

\begin{definition}[$\sigma$-derivation]\label{def:sigma-derivation} Let $R$ be a unital, non-associative ring where $\sigma$ is a unital endomorphism and $\delta$ an additive map on $R$. Then $\delta$ is called a \emph{$\sigma$-derivation} if $\delta(a\cdot b)=\sigma(a)\cdot\delta(b)+\delta(a)\cdot b$ holds for all $a,b\in R$. If $\sigma=\mathrm{id}_R$, then $\delta$ is simply a \emph{derivation}.
\end{definition}

\begin{remark}\label{re:sigma-derivation}
If $\delta$ is a $\sigma$-derivation on a unital, non-associative ring $R$, then $\delta(1)=0$.
\end{remark}

\begin{lemma}[Properties of $\pi$ functions]\label{lem:lem-of-pi}Let $R$ be a unital, non-associative ring, $\sigma$ a unital endomorphism and $\delta$ a $\sigma$-derivation on $R$. Then, in $R[X;\sigma,\delta]$, the following hold for all $a,b\in R$ and $l,m,n\in\mathbb{N}$:
\begin{enumerate}[label=(\roman*)]
\item\label{eq:nystedt-sum}$\sum_{i\in\mathbb{N}}\pi_i^m\left(a\cdot\pi_{l-i}^n(b)\right)=\sum_{i\in\mathbb{N}}\pi_i^m(a)\cdot\pi_l^{i+n}(b).$
\item\label{eq:pi-function-sum-split}$\pi_l^{m+1}=\pi_{l-1}^m\circ\sigma + \pi_l^m\circ\delta =\sigma\circ\pi_{l-1}^m + \delta\circ\pi_l^m$.
\end{enumerate}
\end{lemma}

\begin{proof}A proof of \ref{eq:nystedt-sum} in the associative setting can be found in \cite{Nys13}. However, the proof makes no use of associativity, so we can conclude that \ref{eq:nystedt-sum} holds in the non-associative setting as well.

Regarding \ref{eq:pi-function-sum-split}, we first recall that $\pi_{l}^{m+1}$ consists of the sum of all $\binom{m+1}{l}$ possible compositions of $l$ copies of $\sigma$ and $m+1-l$ copies of $\delta$. Therefore, we can split the sum into a part containing $\sigma$ innermost (outermost) and a part containing $\delta$ innermost (outermost). When $l=0$, we immediately see that the result holds as $\pi_{-1}^m:=0$. When $l>m$, $\pi_l^m:=0$, and in case also $l>m+1$, $\pi_l^{m+1}=\pi_{l-1}^m:=0$. In case $l=m+1$, $\pi_l^l=\pi_{l-1}^{l-1}\circ\sigma=\sigma\circ\pi_{l-1}^{l-1}$, so we can conclude that \ref{eq:pi-function-sum-split} holds when $l=0$ and when $l>m$. For the remaining case $1\leq l\leq m$, we use the recursive formula for binomial coefficients $\binom{m+1}{l}=\binom{m}{l-1}+\binom{m}{l}$ and simply count the terms in the two parts of the sum.
\end{proof}

For any unital, hom-associative ring $R$ with twisting map $\alpha$, we extend $\alpha$ \emph{homogeneously} to an additive map on $R[X;\sigma,\delta]$ by putting $\alpha\left(\sum_{i\in\mathbb{N}}a_iX^i\right):=\sum_{i\in\mathbb{N}}\alpha(a_i)X^i$ for any $a_i\in R$. The next proposition makes use of this construction.

\begin{proposition}[Sufficient conditions for hom-associativity of {$R[X;\sigma,\delta]$} \cite{BRS18}]\label{prop:extended-ore} Let $R$ be a unital, hom-associative ring with twisting map $\alpha$, $\sigma$ a unital endomorphism and $\delta$ a $\sigma$-derivation that both commute with $\alpha$. Extend $\alpha$ homogeneously to $R[X;\sigma,\delta]$. Then $R[X;\sigma,\delta]$ is a unital, hom-associative Ore extension with twisting map $\alpha$.
\end{proposition}

\section{Hom-module theory}\label{sec:hom-module-theory}
In this section, we develop the theory of hom-modules over non-unital, hom-associative rings. Most of the proofs are nearly identical to the classical proofs of the associative case. We have, however, provided them for completeness.

\subsection{Basic definitions and theorems}
\begin{definition}[Hom-module]\label{def:hom-module}
Let $R$ be a non-unital, hom-associative ring with twisting map $\alpha_R$, multiplication written with juxtaposition. Let $M$ be an additive group with a group homomorphism $\alpha_M\colon M\to M$, also called a twisting map. A \emph{right $R$-hom-module} $M_R$ consists of $M$ and an operation $\cdot\colon M\times R\to M$, called \emph{scalar multiplication}, such that for all $r_1,r_2\in R$ and $m_1,m_2\in M$, the following hold:
\begin{align}
(m_1+m_2)\cdot r_1&=m_1\cdot r_1+m_2\cdot r_1 &\text{(right-distributivity)},\tag{M1}\label{eq:hom-mod-1}\\
m_1\cdot(r_1+r_2)&= m_1\cdot r_1+ m_1\cdot r_2 &\text{(left-distributivity)},\tag{M2}\label{eq:hom-mod-2}\\
\alpha_M(m_1)\cdot(r_1r_2)&=(m_1\cdot r_1)\cdot\alpha_R(r_2) &\text{(hom-associativity)}.\tag{M3}\label{eq:hom-mod-3}
\end{align}
A \emph{left $R$-hom-module} is defined analogously and written $_RM$.
\end{definition}

For the sake of brevity, we also allow ourselves to write $M$ in case it does not matter whether it is a right or a left $R$-hom-module, and simply call it an $R$-hom-module. Furthermore, any two right (left) $R$-hom-modules are assumed to be equipped with the same twisting map $\alpha_R$ on $R$.

\begin{remark}\label{re:ring-as-module}A hom-associative ring $R$ is both a right $R$-hom-module $R_R$ and a left $R$-hom-module $_RR$.
\end{remark}

\begin{definition}[Morphism of hom-modules] A \emph{morphism} from a right (left) $R$-hom-module $M$ to a right (left) $R$-hom-module $M'$ is an additive map $f\colon M\to M'$ such that $f\circ\alpha_M=\alpha_{M'}\circ f$ and $f(m\cdot r)=f(m)\cdot r$ ($f(r\cdot m)=r\cdot f(m)$) hold for all $m\in M$ and $r\in R$. If $f$ is also bijective, the two are \emph{isomorphic}, written $M\cong M'$.
\end{definition}

\begin{definition}[Hom-submodule] Let $M$ be a right (left) $R$-hom-module. An \emph{$R$-hom-submodule}, or just \emph{hom-submodule}, is an additive subgroup $N$ of $M$ that is closed under scalar multiplication and invariant under $\alpha_M$. $N$ is then a right (left) $R$-hom-module with twisting maps $\alpha_R$ and $\alpha_N$, the latter being given by the restriction of $\alpha_M$ to $N$. We denote that $N$ is a hom-submodule of $M$ by $N\leq M$ or $M\geq N$, and in case $N$ is a proper subgroup of $M$, by $N<M$ or $M>N$.
\end{definition}

\begin{proposition}[Image and preimage under hom-module morphism]\label{prop:image-pre-image-hom-morphism} Let $f\colon M\to M'$ be a morphism of right (left) $R$-hom-modules, $N\leq M$ and $N'\leq M'$. Then $f(N)$ and $f^{-1}(N')$ are hom-submodules of $M'$ and $M$, respectively.
\end{proposition}

\begin{proof}We see that $f(N)$ and $f^{-1}(N')$ are additive subgroups when considering $f$ as a group homomorphism. Let $r\in R$ and $a'\in f(N)$ be arbitrary. Then there is some $a\in N$ such that $a'=f(a)$, so $a'\cdot r=f(a)\cdot r=f(a\cdot r)\in f(N)$ since $a\cdot r\in N$. Moreover, $\alpha_{M'}(a')=\alpha_{M'}(f(a))=f(\alpha_M(a))=f(\alpha_N(a))\in f(N)$. Now, take any $b\in f^{-1}(N')$. Then there is some $b'\in N'$ such that $f(b)=b'$, so $f(b\cdot r)=f(b)\cdot r= b'\cdot r\in N'$ since $b'\in N'$, and hence $b\cdot r\in f^{-1}(N')$. At last, $f(\alpha_M(b))=\alpha_{M'}(f(b))=\alpha_{M'}(b')=\alpha_{N'}(b')\in N'$, so $\alpha_M(b)\in f^{-1}(N')$. The left case is analogous.
\end{proof}

\begin{proposition}[Intersection of hom-submodules]\label{prop:hom-submodule-intersection} The intersection of any set of hom-submodules of a right (left) $R$-hom-module is a hom-submodule.
\end{proposition}

\begin{proof}We show the case of right $R$-hom-modules; the case of left $R$-hom-modules is analogous. Let $N=\cap_{i\in I} N_i$ be an intersection of hom-submodules $N_i$ of a right $R$-hom-module $M$, where $I$ is some index set. Take any $a,b\in N$ and $j\in I$. Since $a,b\in N_j$ and $N_j$ is an additive subgroup, $(a-b)\in N_j$, and therefore $(a-b)\in N$. For any $r\in R$, $a\cdot r\in N_j$ since $N_j$ is a hom-submodule, and therefore $a\cdot r\in N$. At last, $\alpha_M(a)=\alpha_{N_j}(a)\in N_j$ for the same reason, so $\alpha_M(N)$ is a subset of $N$.
\end{proof}

\begin{definition}[Generating set of hom-submodule]\label{def:generating-set} Let $S$ be a nonempty subset of a right (left) $R$-hom-module $M$. The intersection $N$ of all hom-submodules of $M$ that contain $S$ is called the \emph{hom-submodule generated by $S$}, and $S$ is called a \emph{generating set} of $N$. If there is a finite generating set of $N$, then $N$ is called \emph{finitely generated}.  
\end{definition}

\begin{remark}\label{re:generating-set} The hom-submodule $N$ of a right (left) $R$-hom-module $M$ generated by a nonempty subset $S$  is the smallest hom-submodule of $M$ that contains $S$ in the sense that any other hom-submodule of $M$ that contains $S$ also contains $N$.
\end{remark}

\begin{proposition}[Union of hom-submodules in an ascending chain]\label{prop:hom-submodules-union} Let $M$ be a right (left) $R$-hom-module, and consider an ascending chain $N_1 \leq N_2 \leq \dots$ of hom-submodules of $M$. Then $\cup_{i=1}^\infty N_i$ is a hom-submodule of $M$.
\end{proposition}

\begin{proof} Denote $\cup_{i=1}^\infty N_i$ by $N$, and let $a,b\in N$. Then $a\in N_j$ and $b\in N_k$ for some $j,k\in\mathbb{N}_{>0}$, and since $N_j\leq N_{\max(j,k)}$ and $N_k\leq N_{\max(j,k)}$, we have $a,b\in N_{\max(j,k)}$. Hence $(a-b)\in N_{\max(j,k)}\subseteq N$, so $(a-b)\in N$. Take any $r\in R$. Then, since $a\in N_j$, $a\cdot r\in N_j\subseteq N$, so $a\cdot r\in N$ for the right case, and analogously for the left case. Finally, $\alpha_M(a)=\alpha_{N_j}(a)\in N_j\subseteq N$, so $N$ is invariant under $\alpha_M$.
\end{proof}

\begin{proposition}[Sum of hom-submodules]\label{prop:hom-submodule-sum} Let $M$ be a right (left) $R$-hom-module and $N_1, N_2,\dots, N_k$ any finite number of hom-submodules of $M$. Then $\sum_{i=1}^k N_i=N_1+N_2+\dots+N_k$ is a hom-submodule of $M$.
\end{proposition}

\begin{proof}We prove the right case; the left case is analogous. Let $N:=\sum_{i=1}^kN_i$ and take any $r\in R$, $a_i,b_i\in N_i$. Then $\left(\sum_{i=1}^ka_i\right)\cdot r=\sum_{i=1}^ka_i\cdot r\in N$, and $\sum_{i=1}^k a_i-\sum_{i=1}^kb_i=\sum_{i=1}^k(a_i-b_i)\in N$. At last, $N$ is invariant under $\alpha_N:=\alpha_M\vert_N$ since $\alpha_M\left(\sum_{i=1}^ka_i\right)=\sum_{i=1}^k \alpha_M(a_i)=\sum_{i=1}^k\alpha_{N_1}(a_i)\in N$.
\end{proof}

\begin{corollary}[The modular law for hom-modules]\label{cor:modular-law} Let $M$ be a right (left) $R$-hom-module, and $M_1,M_2,$ and $M_3$ hom-submodules of $M$ with $M_3\leq M_1$. Then the \emph{modular law}
$\left(M_1\cap M_2\right)+M_3=M_1\cap(M_2+M_3)$ holds.
\end{corollary}

\begin{proof} The modular law holds for $M_1,M_2$ and $M_3$ when considered as additive groups. By \autoref{prop:hom-submodule-intersection} and \autoref{prop:hom-submodule-sum}, the intersection and sum of any two hom-submodules of $M$ are also hom-submodules of $M$, and hence the modular law holds for $M_1, M_2$ and $M_3$ as hom-modules as well.
\end{proof}

\begin{proposition}[Direct sum of hom-modules]\label{prop:direct-sum} Let $M_1, M_2,\dots, M_k$ be any finite number of right $R$-hom-modules. Endowing the (external) direct sum $M:=\bigoplus_{i=1}^k M_i=M_1\oplus M_2\oplus\dots\oplus M_k$ with the following scalar multiplication and twisting map on $M$, makes it a right $R$-hom-module: 
\begin{align*}
&\bullet\colon M\times R\to M, &(m_1,m_2,\dots,m_k)\bullet r&:=(m_1\cdot r,m_2\cdot r,\dots,m_k\cdot r),&\\
&\alpha_M\colon M\to M, &\alpha_M((m_1,m_2, \dots,m_k))&:=\left(\alpha_{M_1}(m_1),\alpha_{M_2}(m_2), \dots,\alpha_{M_k}(m_k)\right).&
\end{align*}
Here, $(m_1,m_2,\dots,m_k)\in M$, $r\in R$, and $\alpha_{M_{i}}$ is the twisting map on $M_i$ for $1\leq i\leq k$. 
\end{proposition}

\begin{proof}Since $M$ is an additive group, what is left to check is that $\alpha_M$ is a group homomorphism, i.e. an additive map, and that $\eqref{eq:hom-mod-1}, \eqref{eq:hom-mod-2}$ and $\eqref{eq:hom-mod-3}$ in \autoref{def:hom-module} holds. This is readily verified by routine calculations. 
\end{proof}
An analogous result holds for left $R$-hom-modules.

\begin{corollary}[Associativity of the direct sum]\label{cor:isomorphic-direct-sum} For any right (left) $R$-hom-modules $M_1,M_2$, and $M_3$, $(M_1\oplus M_2)\oplus M_3\cong M_1\oplus M_2 \oplus M_3\cong M_1\oplus (M_2 \oplus M_3)$.
\end{corollary}

\begin{proof}We prove the right case of the first isomorphism. The proof of the second isomorphism is similar, as are all the left cases. Considered as additive groups, $M:=(M_1\oplus M_2)\oplus M_3\cong M_1\oplus M_2\oplus M_3=:M'$ by the natural isomorphism $f(((m_1,m_2),m_3))=(m_1,m_2,m_3)$ for any $((m_1,m_2),m_3)\in M$. Let $r\in R$ be arbitrary. Then 
\begin{align*}
&f(((m_1,m_2),m_3)\bullet r)=f(((m_1,m_2)\bullet r,m_3\cdot r))=f(((m_1\cdot r,m_2\cdot r),m_3\cdot r))\\
&=(m_1\cdot r,m_2\cdot r, m_3\cdot r)=f(((m_1,m_2),m_3))\bullet r,\\
&f(\alpha_M((m_1,m_2),m_3))=f(((\alpha_{M_1}(m_1),\alpha_{M_2}(m_2)),\alpha_{M_3}(m_3)))\\
&=(\alpha_{M_1}(m_1),\alpha_{M_2}(m_2),\alpha_{M_3}(m_3))=\alpha_{M'}(f(((m_1,m_2),m_3))).&\qedhere
\end{align*}
\end{proof}

\begin{proposition}[Quotient hom-module] Let $M_R$ be a right $R$-hom-module with twisting map $\alpha_M$ on $M$. Let $N_R\leq M_R$ and consider the additive groups $M$ and $N$ of $M_R$ and $N_R$, respectively. Form the quotient group $M/N$ with elements of the form $m+N$ for $m\in M$. Then $M/N$ becomes a right $R$-hom-module when endowed with the following twisting map and scalar multiplication for $m\in M$ and $r\in R$:
\begin{align*}
&\bullet\colon  M/N\times R\to M/N, & (m+N)\bullet r&:=m\cdot r+N,&\\
&\alpha_{M/N}\colon M/N \to M/N, & \alpha_{M/N}(m+N)&:=\alpha_M(m)+N.&
\end{align*}
\end{proposition}

\begin{proof}First, let us make sure that the scalar multiplication and twisting map are both well-defined. To this end, take two arbitrary elements of $M/N$. They are of the form $m_1+N$ and $m_2+N$ for some $m_1,m_2\in M$. If $m_1+N=m_2+N$, then $(m_1-m_2)\in N$, and since $N_R$ is a right $R$-hom-module, $(m_1-m_2)\cdot r_1\in N$ for any $r_1\in R$. Then $(m_1\cdot r_1 - m_2\cdot r_1)\in N$, so $m_1\cdot r_1 + N = m_2\cdot r_1 + N$, and hence $(m_1+N)\bullet r_1=(m_2+N)\bullet r_1$, so the scalar multiplication is well-defined. Now, since $(m_1-m_2)\in N$, $\alpha_M(m_1-m_2)\in N$ due to the fact that $N_R\leq M_R$. On the other hand, $\alpha_M(m_1-m_2)=\alpha_M(m_1)-\alpha_M(m_2)$, so $\left(\alpha_M(m_1)-\alpha_M(m_2)\right)\in N$. Then $\alpha_M(m_1)+N=\alpha_M(m_2)+N$, and therefore $\alpha_{M/N}(m_1+N)=\alpha_{M/N}(m_2+N)$, which proves that $\alpha_{M/N}$ is well-defined. Furthermore, $\alpha_{M/N}$ is a group homomorphism since for any $(m_3+N),(m_4+N)\in M/N$ where $m_3,m_4\in M$,
\begin{align*}
&\alpha_{M/N}\left(\left(m_3+N\right)+\left(m_4+N\right)\right)=\alpha_{M/N}\left((m_3+m_4)+N\right)=\alpha_M(m_3+m_4)+N\\
&=\left(\alpha_M(m_3)+\alpha_M(m_4)\right)+N=\left(\alpha_M(m_3)+N\right)+\left(\alpha_M(m_4)+N\right)\\
&=\alpha_{M/N}\left(m_3+N\right)+\alpha_{M/N}\left(m_4+N\right).
\end{align*}
By straightforward calculations, one readily verifies that the three hom-module axioms \eqref{eq:hom-mod-1}, \eqref{eq:hom-mod-2}, and \eqref{eq:hom-mod-3} in \autoref{def:hom-module} hold. 
\end{proof}
Again, an analogous result holds for left $R$-hom-modules as well.

\begin{corollary}[The natural projection]\label{cor:natural-projection} Let $M$ be a right (left) $R$-hom-module with $N\leq M$. Then the \emph{natural projection} $\pi\colon M\to M/N$ defined by $\pi(m)= m+N$ for any $m\in M$ is a surjective morphism of hom-modules.
\end{corollary}

\begin{proof} We know that $\pi$ is a surjective group homomorphism, and for any $m\in M$ and $r\in R$, $\pi(m\cdot r)=m\cdot r+N=(m+N)\bullet r=\pi(m)\bullet r$ for the right case, and analogously for the left case. We also have that $\pi(\alpha_M(m))=\alpha(m)+N=\alpha_{M/N}(m+N)=\alpha_{M/N}(\pi(m))$, completing the proof.
\end{proof}

\begin{corollary}[Hom-submodules of quotient hom-modules]\label{cor:quotient-submodule} Let $M$ be a right (left) $R$-hom-module with $N\leq M$. If $L$ is a hom-submodule of $M/N$, then $L=K/N$ for some hom-submodule $K$ of $M$ that contains $N$.
\end{corollary}

\begin{proof}Let $L$ be a hom-submodule of $M/N$. Using the natural projection $\pi\colon M\to M/N$ from \autoref{cor:natural-projection}, we know that $K=\pi^{-1}(L)$ is a hom-submodule of $M$ since it is the preimage of a morphism of hom-submodules (cf. \autoref{prop:image-pre-image-hom-morphism}). By the surjectivity of $\pi$, $\pi(K)=\pi(\pi^{-1}(L))=L$, so $L=\pi(K)=K/N$.
\end{proof}

\begin{theorem}[The first isomorphism theorem for hom-modules]\label{thm:first-isomorphism-thm} Let $f\colon M\to M'$ be a morphism of right (left) $R$-hom-modules. Then $\ker f$ is a hom-submodule of $M$, $\im f$ is a hom-submodule of $M'$, and $M/\ker f\cong \im f$.
\end{theorem}

\begin{proof}We prove the case of right $R$-hom-modules; the case of left $R$-hom-modules is analogous. By definition, $\ker f$ is the preimage of the hom-submodule $0$ of $M'$, and hence it is a hom-submodule of $M$ by \autoref{prop:image-pre-image-hom-morphism}. Now, $\im f=f(M)$, so by the same proposition, $\im f$ is a hom-submodule of $M'$. The map $g\colon M/\ker f\to \im f$ defined by $g(m+\ker f)=f(m)$ for any $(m+\ker f)\in M/\ker f$ is a well-defined group isomorphism. Furthermore, $g((m+\ker f)\bullet r)=g(m\cdot r+\ker f)=f(m\cdot r)=f(m)\cdot r=g(m+\ker f)\cdot r$. At last, $g(\alpha_{M/\ker f}(m+\ker f))=g(\alpha_M(m)+\ker f)=f(\alpha_M(m))=\alpha_{M'}(f(m))=\alpha_{\im f}(f(m))=\alpha_{\im f}(g(m+\ker f))$, which completes the proof.
\end{proof}

\begin{theorem}[The second isomorphism theorem for hom-modules]\label{thm:second-isomorphism-thm} Let $M$ be a right (left) $R$-hom-module with $N\leq M$ and $L\leq M$. Then $N/(N\cap L)\cong (N+L)/L$.
\end{theorem}

\begin{proof}By \autoref{prop:hom-submodule-intersection}, $N\cap L$ is a hom-submodule of $N$ and by \autoref{prop:hom-submodule-sum}, $N+L$ is a hom-module with $L=(0+L)\leq (N+L)$, so the expression makes sense. The map $f\colon N\to (N+L)/L$ defined by $f(n)=n+L$ for any $n\in N$ is a group homomorphism. Furthermore, it is surjective since for any $((n+l)+L)\in (N+L)/L$. We have that $(n+l)+L=(n+L)+(l+L)=n+L+(0+L)=n+L=f(n)$. For any $r\in R$, $f(n\cdot r)=n\cdot r+L=(n+L)\bullet r= f(n)\bullet r$ (similarly for the left case), and moreover, $f(\alpha_N(n))=\alpha_N(n)+L=(\alpha_N(n)+\alpha_L(0))+L=\alpha_{N+L}(n+0)+L=\alpha_{(N+L)/L}(n+L)=\alpha_{(N+L)/L}(f(n))$. We also see that $\ker f=N\cap L$, so by \autoref{thm:first-isomorphism-thm}, $N/(N\cap L)\cong (N+L)/L$.
\end{proof}

\begin{theorem}[The third isomorphism theorem for hom-modules]\label{thm:third-isomorphism-thm} Let $M$ be a right (left) $R$-hom-module with $L\leq N\leq M$. Then $N/L$ is a hom-submodule of $M/L$ and $(M/L)/(N/L)\cong M/N$.
\end{theorem}

\begin{proof}According to \autoref{cor:natural-projection}, the natural projection $\pi\colon M\to M/L$ is a morphism of right (left) hom-modules, so hom-submodules of $M$ are mapped to hom-submodules of $M/L$. Since $N\leq M$, $N/L=\pi(N)\leq\pi(M)=M/L$, using that $\pi$ is surjective. The map $f\colon M/L\to M/N$ defined by $f(m+L)=m+N$ for any $(m+L)\in M/L$ is a well-defined surjective group homomorphism. Moreover, for any $r\in R$, $f((m+L)\bullet r)=f(m\cdot r+L)=m\cdot r+N=(m+N)\bullet r=f(m+L)\bullet r$ (analogously for the left case), and $f(\alpha_{M/L}(m+L))=f(\alpha_M(m)+L)=\alpha_M(m)+N=\alpha_{M/N}(m+N)=\alpha_{M/N}(f(m+L))$. We also see that $\ker f=N/L$, so using \autoref{thm:first-isomorphism-thm}, $(M/L)/\ker f=(M/L)/(N/L)\cong \im f= M/N$.
\end{proof}

\subsection{The hom-Noetherian conditions}
Recall that a family $\mathcal{F}$ of subsets of a set $S$ satisfies the \emph{ascending chain condition} if there is no properly ascending infinite chain $S_1\subset S_2\subset\ldots$ of subsets from $\mathcal{F}$. Furthermore, an element in $\mathcal{F}$ is called a \emph{maximal element} of $\mathcal{F}$ provided there is no element of $\mathcal{F}$ that properly contains that element.

\begin{proposition}[The hom-Noetherian conditions for hom-modules]\label{prop:hom-noetherian-conditions-for-modules} Let $M$ be a right (left) $R$-hom-module. Then the following conditions are equivalent:
\begin{enumerate}[label=(NM\arabic*)]
	\item\label{eq:hom-noetherian-module-1}$M$ satisfies the ascending chain condition on its hom-submodules.
	\item\label{eq:hom-noetherian-module-2} Any nonempty family of hom-submodules of $M$ has a maximal element.
	\item\label{eq:hom-noetherian-module-3} Any hom-submodule of $M$ is finitely generated.
\end{enumerate}
\end{proposition}

\begin{proof} The following proof is an adaptation of a proof that can be found in \cite{GW04} to the hom-associative setting.

\ref{eq:hom-noetherian-module-1}$\implies$\ref{eq:hom-noetherian-module-2}: Let $\mathcal{F}$ be a nonempty family of hom-submodules of $M$ that does not have a maximal element and pick an arbitrary hom-submodule $S_1$ in $\mathcal{F}$. Since $S_1$ is not a maximal element, there exists some $S_2\in\mathcal{F}$ such that $S_1<S_2$. Now, $S_2$ is not a maximal element either, so there exists some $S_3\in\mathcal{F}$ such that $S_2<S_3$. Continuing in this manner we get an infinite chain of hom-submodules $S_1<S_2<\dots$, which proves the contrapositive statement.

\ref{eq:hom-noetherian-module-2}$\implies$\ref{eq:hom-noetherian-module-3}: Assume \ref{eq:hom-noetherian-module-2} holds, let $N$ be an arbitrary hom-submodule of $M$, and $\mathcal{G}$ the family of all finitely generated hom-submodules of $N$. Since the zero module is a hom-submodule of $N$ that is finitely generated, $\mathcal{G}$ is clearly nonempty and thus contains a maximal element $L$ by assumption. If $N=L$, we are done, so assume the opposite and take some $n\in N\backslash L$. Now, let $K$ be the hom-submodule of $N$ generated by the set $L\cup\{n\}$. Then $K$ is finitely generated as well, so $K\in \mathcal{G}$. Moreover, $L<K$, which is a contradiction since $L$ is a maximal element in $\mathcal{G}$. Therefore, $N=L$, and $N$ is finitely generated. 

\ref{eq:hom-noetherian-module-3}$\implies$\ref{eq:hom-noetherian-module-1}: Assume \ref{eq:hom-noetherian-module-3} holds, let $T_1\leq T_2\leq\dots$ be an ascending chain of hom-submodules of $M$, and $T=\cup_{i=1}^\infty T_i$. By \autoref{prop:hom-submodules-union}, $T$ is a hom-submodule of $M$, and hence it is finitely generated by some set $S$ which by \autoref{def:generating-set} is contained in $T$. Moreover, since $S$ is finite, it needs to be contained in $T_j$ for some $j\in\mathbb{N}_{>0}$. However, $T_j=T$ by \autoref{re:generating-set}, so $T_k=T_j$ for all $k\geq j$, and hence the ascending chain condition holds.
\end{proof}

\begin{definition}[Hom-Noetherian module]A right (left) $R$-hom-module is called \emph{hom-Noetherian} if it satisfies the three equivalent conditions of \autoref{prop:hom-noetherian-conditions-for-modules} on its hom-submodules.
\end{definition}

Appealing to \autoref{re:ring-as-module}, all properties that hold for right (left) hom-modules necessarily also hold for hom-associative rings, replacing ``hom-submodule" by ``right (left) hom-ideal''. Hence we have the following:

\begin{corollary}[The hom-Noetherian conditions for hom-associative rings]\label{prop:hom-noetherian-conditions-for-rings} Let $R$ be a non-unital, hom-associative ring. Then the following conditions are equivalent:
\begin{enumerate}[label=(NR\arabic*)]
	\item\label{eq:hom-noetherian-ring-1} $R$ satisfies the ascending chain condition on its right (left) hom-ideals.
	\item\label{eq:hom-noetherian-ring-2} Any nonempty family of right (left) hom-ideals of $R$ has a maximal element.
	\item\label{eq:hom-noetherian-ring-3} Any right (left) hom-ideal of $R$ is finitely generated.
\end{enumerate}
\end{corollary}

\begin{definition}[Hom-Noetherian ring]A non-unital, hom-associative ring $R$ is called \emph{right (left) hom-Noetherian} if it satisfies the three equivalent conditions of \autoref{prop:hom-noetherian-conditions-for-rings} on its right (left) hom-ideals. If $R$ satisfies the conditions on both its right and its left hom-ideals, it is called \emph{hom-Noetherian}.
\end{definition}

\begin{remark}\label{re:noetherian-ring} If the twisting map is either the identity map or the zero map, a right (left) hom-Noetherian ring is simply a right (left) Noetherian ring (cf. \autoref{re:ideal}). If $R$ is a unital, hom-associative ring, then all right (left) ideals of $R$ are actually right (left) hom-ideals; if $I$ is a right ideal of $R$ and $i\in I$, then $\alpha(i)=\alpha(i)\cdot(1\cdot 1)=i\cdot\alpha(1)\in I$, and similarly for the left case. In particular, $R$ is right (left) hom-Noetherian if and only if $R$ is right (left) Noetherian.
\end{remark}

\begin{proposition}[Surjective hom-Noetherian hom-module morphism]\label{prop:surjective-hom-noetherian-morphism}The hom-Noetherian conditions are preserved by surjective morphisms of right (left) $R$-hom-modules.
\end{proposition}

\begin{proof}It is sufficient to prove this for any of the three equivalent conditions \ref{eq:hom-noetherian-module-1}, \ref{eq:hom-noetherian-module-2}, or \ref{eq:hom-noetherian-module-3} in \autoref{prop:hom-noetherian-conditions-for-modules}, so let us choose \ref{eq:hom-noetherian-module-2}. To this end, let $f\colon M\to M'$ be a surjective morphism of right (left) $R$-hom-modules where $M$ is hom-Noetherian. Let $\mathcal{F}'$ be a nonempty family of right (left) hom-submodules of $M'$. Now, consider the corresponding family in $M$, $\mathcal{F}=\{f^{-1}(N')\colon N'\in \mathcal{F}'\}$. By the surjectivity of $f$, this family is nonempty, and since $M$ is Noetherian, it has a maximal element $f^{-1}(N'_0)$ for some $N'_0\in\mathcal{F}'$. We would like to show that $N'_0$ is a maximal element of $\mathcal{F}'$. Assume there exists an element $N'\in \mathcal{F}'$ such that $N'_0<N'$. We know that the operation of taking preimages under any function preserves inclusions sets. We also know that the preimage of any hom-submodule is again a hom-submodule by \autoref{prop:image-pre-image-hom-morphism}, so taking preimages under a hom-morphism preserves the inclusions on the hom-submodules, and therefore $N'_0<N'$ implies that $f^{-1}(N'_0)<f^{-1}(N')$, which contradicts the maximality of $f^{-1}(N'_0)$ in $\mathcal{F}$. Hence $N'_0$ is a maximal element of $\mathcal{F}'$, and $M'$ is hom-Noetherian.
\end{proof}

\begin{proposition}[Hom-Noetherian condition on quotient hom-module]\label{prop:quotient-hom-noetherian} Let $M$ be a right (left) $R$-hom-module, and $N\leq M$. Then $M$ is hom-Noetherian if and only if $M/N$ and $N$ are hom-Noetherian.
\end{proposition}

\begin{proof}This is again an adaptation of a proof that can be found in \cite{GW04} to the hom-associative setting.

$(\Longrightarrow)\colon$ Assume $M$ is hom-Noetherian and $N\leq M$. Then any hom-submodule of $N$ is also a hom-submodule of $M$, and hence it is finitely generated, and $N$ therefore also hom-Noetherian. If $L_1\leq L_2\leq\dots$ is an ascending chain of hom-submodules of $M/N$, then from \autoref{cor:quotient-submodule}, each $L_i=M_i/N$ for some $M_i$ with $N\leq M_i\leq M$. Furthermore, $M_1\leq M_2\leq\dots$, but since $M$ is hom-Noetherian, there is some $n$ such that $M_i=M_n$ for all $i\geq n$. Then $L_i=M_n/N=L_n$ for all $i\geq n$, so $M/N$ is hom-Noetherian.

$(\Longleftarrow)\colon$ Assume $M/N$ and $N$ are hom-Noetherian. Let $S_1\leq S_2\leq\dots$ be an ascending chain of hom-submodules of $M$. By \autoref{prop:hom-submodule-intersection}, $S_i\cap N$ is a hom-submodule of $N$ for every $i\in \mathbb{N}_{>0}$, and furthermore $S_i\cap N\leq S_{i+1}\cap N$. We thus have an ascending chain $S_1\cap N\leq S_2\cap N\leq\dots$ of hom-submodules of $N$. By \autoref{prop:hom-submodule-sum}, $S_i+N$ is a hom-submodule of $M$, and moreover, $N=0+N$ is a hom-submodule of $S_i+N$, so we can consider $\left(S_i+N\right)/N$. Now, $\left(S_i+N\right)/N \leq \left(S_{i+1}+N\right)/N$ by \autoref{cor:quotient-submodule}, so we have an ascending chain $(S_1+N)/N\leq (S_2+N)/N\leq\dots$ of hom-submodules of $M/N$. Since both $N$ and $M/N$ are hom-Noetherian, there is some $k$ such that $S_j\cap N=S_k\cap N$ and $(S_j+N)/N=(S_k+N)/N$ for all $j\geq k$. The latter equation implies that for any $s_j\in S_j$ and $n\in N$, there are $s_k\in S_k$ and $n'\in N$ such that $(s_j+n)+N=(s_k+n')+N$. Hence $x:=((s_j+n)-(s_k+n'))\in N$, and therefore $s_j+n=(s_k+(x+n'))\in (S_k+N)$, so that $(S_j+N)\leq (S_k+N)$, and by a similar argument, $(S_k+N)\leq (S_j+N)$, so $S_j+N=S_k+N$ for all $j\geq k$. Using this and the modular law for hom-modules (\autoref{cor:modular-law}), $S_k=(S_k\cap N)+S_k=(S_j\cap N)+S_k=S_j\cap(N+S_k)=S_j\cap (S_k+N)=S_j\cap (S_j+N)=S_j$ for all $j\geq k$, and hence $M$ is hom-Noetherian.
\end{proof}

\begin{corollary}[Finite direct sum of hom-Noetherian modules] Any finite direct sum of hom-Noetherian modules is hom-Noetherian.
\end{corollary}

\begin{proof} We prove this by induction. 

Base case ($\mathsf{P}(2)$): Let $M_1$ and $M_2$ be two hom-Noetherian modules and consider the direct sum $M=M_1\oplus M_2$, which is a right (left) $R$-hom-module by \autoref{prop:direct-sum}. Moreover, $M_1\cong M_1\oplus0$ as additive groups, and for any $r\in R$, $f((m_1,0)\bullet r)=f((m_1\cdot r, 0\cdot r))=f((m_1\cdot r,0))=m_1\cdot r=f((m_1,0))\cdot r$. Now, $f(\alpha_{M_1\oplus 0}((m_1,0)))=f((\alpha_{M_1}(m_1),0))=\alpha_{M_1}(m_1)=\alpha_{M_1}(f(m_1,0))$, so as right (left) $R$-hom-modules, $M_1\cong M_1\oplus 0\leq M$. Similarly, the projection $g\colon M\to M_2$ is a surjective morphism of right (left) $R$-hom-modules with $\ker g=M_1\oplus 0$, so by \autoref{thm:first-isomorphism-thm}, $M/(M_1\oplus 0)\cong M_2$. Due to \autoref{prop:surjective-hom-noetherian-morphism}, $M_1\oplus 0$ and $M/(M_1\oplus 0)$ are both hom-Noetherian, and so by \autoref{prop:quotient-hom-noetherian}, $M$ is hom-Noetherian.

Induction step ($\forall k\in\mathbb{N}_{>1}\ (\mathsf{P}(k)\rightarrow\mathsf{P}(k+1))$): Assume $M'=\bigoplus_{i=1}^k M_i$ is hom-Noetherian for $2\leq k$, where each $M_i$ is a hom-Noetherian right (left) $R$-hom-module. Let $M_{k+1}$ be a hom-Noetherian right (left) $R$-hom-module. Then $\bigoplus_{i=1}^{k+1}M_i\cong M'\oplus M_{k+1}$ by \autoref{cor:isomorphic-direct-sum}.  The latter of the two is hom-Noetherian by the base case, and by \autoref{prop:surjective-hom-noetherian-morphism} the former as well. 
\end{proof}

\section{Hilbert's basis theorem for hom-associative Ore extensions}\label{sec:hilbert-theorem}
In this section, we consider unital, non-associative Ore extensions $R[X;\sigma,\delta]$ over some unital, non-associative ring $R$. First, recall from \autoref{prop:extended-ore} that if $R$ is hom-associative, then a sufficient condition for $R[X;\sigma,\delta]$ to be hom-associative is that $\sigma$ is a unital endomorphism and $\delta$ a $\sigma$-derivation that both commute with the twisting map $\alpha$ of $R$, the latter extended homogeneously to $R[X;\sigma,\delta]$. Moreover, from \autoref{re:noetherian-ring}, for unital, hom-associative rings, being right (left) hom-Noetherian is the same as being right (left) Noetherian. Also recall that the \emph{associator} is the map $(\cdot,\cdot,\cdot)\colon R\times R\times R\to R$ defined by $(r,s,t)=(r\cdot s)\cdot t-r\cdot (s\cdot t)$ for any $r,s,t\in R$. The \emph{left}, \emph{middle}, and \emph{right nucleus} of $R$ are denoted by $N_l(R)$, $N_m(R)$, and $N_r(R)$, respectively. As sets, they are defined as $N_l(R):=\{r\in R\colon (r,s,t)=0,\ s,t\in R\}$, $N_m(R):=\{s\in R\colon (r,s,t)=0,\ r,t\in R\}$, and $N_r(R):=\{t\in R\colon (r,s,t)=0,\ r,s\in R\}$. The \emph{nucleus} of $R$, written $N(R)$, is defined as the set $N(R):=N_l(R)\cap N_m(R)\cap N_r(R)$. By the \emph{associator identity} $u\cdot(r,s,t)+(u,r,s)\cdot t+(u,r\cdot s,t)=(u\cdot r,s,t)+(u,r,s\cdot t)$, holding for all $r,s,t,u\in R$, $N_r(R), N_m(R), N_r(R)$, and hence also $N(R)$, are all associative subrings of $R$.

\begin{proposition}[Associator of $X^k$]\label{prop:X-nucleus}Let $R[X;\sigma,\delta]$ be a unital, non-associative Ore extension of a unital, non-associative ring $R$. Assume $\sigma$ is a unital endomorphism and $\delta$ a $\sigma$-derivation on $R$. Then $X^k\in N(R[X;\sigma,\delta])$ for any $k\in\mathbb{N}$.
\end{proposition}

\begin{proof}By identifying $X^0$ with $1\in R$, $X^0\in N(R[X;\sigma,\delta])$. We now wish to show that $X\in N(R[X;\sigma,\delta])$. In order to do so, we must show that $X$ associates with all polynomials in $R[X;\sigma,\delta]$. Due to distributivity, it is however sufficient to prove that $X$ associates with arbitrary monomials $aX^m$ and $bX^n$ in $R[X;\sigma,\delta]$. To this end, first note that $aX^m\cdot X=\sum_{i\in\mathbb{N}} \left(a\cdot\pi_i^m(1)\right)X^{i+1}=aX^{m+1}$ since $\sigma$ is unital by assumption, and $\delta(1)=0$ by \autoref{re:sigma-derivation}. Then,
\allowdisplaybreaks
\begin{align*}
&\left(aX^m\cdot bX^n\right)\cdot X=\left(\sum_{i\in\mathbb{N}}\left(a\cdot\pi_i^m(b)\right)X^{i+n}\right)\cdot X=\sum_{i\in\mathbb{N}}\left(\left(a\cdot\pi_i^m(b)\right)X^{i+n}\right)\cdot X\\
&=\sum_{i\in\mathbb{N}}\left(a\cdot\pi_i^m(b)\right)X^{i+n+1}=aX^m\cdot bX^{n+1}= aX^m\cdot \left(bX^n\cdot X\right),
\end{align*}
so $X\in N_r(R[X;\sigma,\delta])$. Also, by using \ref{eq:pi-function-sum-split} in \autoref{lem:lem-of-pi},
\begin{align*}
&\left(aX^m\cdot X\right)\cdot bX^n=aX^{m+1}\cdot bX^n=\sum_{i\in\mathbb{N}}\left(a\cdot\pi_i^{m+1}(b)\right)X^{i+n}\\
&=\sum_{i\in\mathbb{N}}\left(a\cdot\left(\pi_{i-1}^m\circ\sigma(b)+\pi_i^m\circ\delta(b)\right)\right)X^{i+n}\\
&=\sum_{j\in\mathbb{N}}\left(a\cdot\pi_j^m(\sigma(b))\right)X^{j+n+1}+\sum_{i\in\mathbb{N}}\left(a\cdot\pi_i^m(\delta(b))\right)X^{i+n}\\
&=aX^m\cdot\sigma(b)X^{n+1}+aX^m\cdot\delta(b)X^n=aX^m\cdot\left(\sigma(b)X^{n+1}+\delta(b)X^n\right)\\
&=aX^m\cdot\sum_{i\in\mathbb{N}}\left(1\cdot\pi_i^1(b)\right)X^{n+i}=aX^m\cdot\left(X\cdot bX^n\right),
\end{align*}  
so $X \in N_m(R[X;\sigma,\delta])$. By a similar calculation, $X\in N_l(R[X;\sigma,\delta])$, so $X\in N(R[X;\sigma,\delta])$. Since $N(R[X;\sigma,\delta])$ is a ring it also contains all powers of $X$, so $X^k\in N(R[X;\sigma,\delta])$ for any $k\in\mathbb{N}$. \qedhere

\end{proof}

\begin{proposition}[Hom-modules of {$R[X;\sigma,\delta]$}]\label{prop:ore-module}Let $R$ be a unital, Noetherian, hom-associative ring with twisting map $\alpha$, $\sigma$ a unital endomorphism and $\delta$ a $\sigma$-derivation that both commute with $\alpha$. Extend $\alpha$ homogeneously to $R[X;\sigma,\delta]$. Then, for any $m\in\mathbb{N}$, $\sum_{i=0}^m X^iR$ ($\sum_{i=0}^m RX^i$) is a hom-Noetherian right (left) $R$-hom-module.
\end{proposition}

\begin{proof} Let us prove the right case; the left case is similar, but slightly simpler. Put $M=\sum_{i=0}^m X^iR$. First note that $M$ really is a subset of $R[X;\sigma,\delta]$, where the elements are of the form $\sum_{i=0}^m 1 X^i\cdot r_i X^0$ with $r_i\in R$. When identifying $1X^i$ with $X^i$ and $r_i$ with $r_iX^0$, this gives us elements of the form $\sum_{i=0}^m X^i\cdot r_i$. Using this identification also allows us to write the multiplication in $R$, which in \autoref{def:hom-module} is done by juxtaposition, by ``$\cdot$'' instead. The purpose of this is do be consistent with our previous notation.

Since distributivity follows from that in $R[X;\sigma,\delta]$, it suffices to show that the multiplication in $R[X;\sigma,\delta]$ is a scalar multiplication, and that we have twisting maps $\alpha_M$ and $\alpha_R$ that give us hom-associativity. To this end, for any $r\in R$ and any element in $M$ (which is of the form described above), by using \autoref{prop:X-nucleus},
\begin{equation}
\left(\sum_{i=0}^m X^i\cdot r_i\right)\cdot r=\sum_{i=0}^m\left(X^i\cdot r_i\right)\cdot r=\sum_{i=0}^m X^i\cdot (r_i\cdot r),\label{eq:ore-module-mult}
\end{equation}
and the latter is clearly an element of $M$. Now, we claim that $M$ is invariant under the homogeneously extended twisting map on $R[X;\sigma,\delta]$. To follow the notation in \autoref{def:hom-module}, let us denote this map when restricted to $M$ by $\alpha_M$, and that of $R$ by $\alpha_R$. Then, by using the additivity of $\alpha_M$ and $\alpha_R$, as well as the fact that the latter commutes with $\delta$ and $\sigma$, we get
\begin{align}
&\alpha_M\left(\sum_{i=0}^mX^i\cdot r_i\right)=\alpha_M\left(\sum_{i=0}^m\sum_{j\in\mathbb{N}}\pi_j^i(r_i)X^j\right)=\sum_{i=0}^m\sum_{j\in\mathbb{N}}\alpha_M\left(\pi_j^i(r_i)X^j\right)\nonumber\\
&=\sum_{i=0}^m\sum_{j\in\mathbb{N}}\alpha_R\left(\pi_j^i(r_i)\right)X^j=\sum_{i=0}^m\sum_{j\in\mathbb{N}}\pi_j^i(\alpha_R(r_i))X^j=\sum_{i=0}^mX^i\cdot \alpha_R(r_i),\label{eq:right-homogeneous-alpha}
\end{align}
which again is an element of $M$. At last, let $r,s\in R$ be arbitrary. Then,
\begin{align*}
&\alpha_M\left(\sum_{i=0}^mX^i\cdot r_i\right)\cdot (r\cdot s)\stackrel{\eqref{eq:right-homogeneous-alpha}}{=}\left(\sum_{i=0}^mX^i\cdot \alpha_R(r_i)\right)\cdot(r\cdot s)\stackrel{\eqref{eq:ore-module-mult}}{=}\sum_{i=0}^mX^i\cdot \left(\alpha_R(r_i)\cdot(r\cdot s)\right)\\
&=\sum_{i=0}^mX^i\cdot \left((r_i\cdot r)\cdot \alpha_R(s)\right)\stackrel{\eqref{eq:ore-module-mult}}{=}\left(\sum_{i=0}^mX^i\cdot (r_i\cdot r)\right)\cdot \alpha_R(s)\\
&\stackrel{\eqref{eq:ore-module-mult}}{=}\left(\left(\sum_{i=0}^mX^i\cdot r_i\right)\cdot r\right)\cdot \alpha_R(s),
\end{align*}
which proves hom-associativity. What is left to prove is that $M$ is hom-Noetherian. Now, let us define $f\colon\bigoplus_{i=0}^m R\to M$ by $(r_0,r_1,\dots,r_m)\mapsto\sum_{i=0}^m X^i\cdot r_i$ for any $(r_0,r_1,\dots,r_m)\in \bigoplus_{i=0}^mR$. We see that $f$ is additive, and for any $r\in R$, we have $f((r_0,r_1,\dots,r_m)\bullet r)=f((r_0,r_1,\dots,r_m))\cdot r$. A similar argument gives $f(\alpha_{\bigoplus_{i=0}^mR}((r_0,r_1,\dots,r_m)))=\alpha_M(f((r_0,r_1,\dots,r_m)))$, which shows that $f$ is a morphism of two right $R$-hom-modules. Moreover, $f$ is surjective, and so by \autoref{prop:surjective-hom-noetherian-morphism}, $M$ is hom-Noetherian.
\end{proof}

\begin{lemma}[Properties of {$R[X;\sigma,\delta]^\text{op}$}]\label{lem:opposite-hom-noetherian}Let $R$ be a unital, hom-associative ring with twisting map $\alpha$, $\sigma$ an automorphism and $\delta$ a $\sigma$-derivation that both commute with $\alpha$. Extend $\alpha$ homogeneously to $R[X;\sigma,\delta]$. Then the following hold:
\begin{enumerate}[label=(\roman*)]
	\item\label{le:Rop-automorphism} $\sigma^{-1}$ is an automorphism on $R^{\text{op}}$ that commutes with $\alpha$.
	\item\label{le:Rop-derivation} $-\delta\circ\sigma^{-1}$ is a $\sigma^{-1}$-derivation on $R^{\text{op}}$ that commutes with $\alpha$.
	\item\label{le:Rop-equality} $R[X;\sigma,\delta]^\text{op}\cong R^{\text{op}}[X;\sigma^{-1},-\delta\circ\sigma^{-1}]$.
\end{enumerate} 
\end{lemma}

\begin{proof}That $\sigma^{-1}$ is an automorphism and $-\delta\circ\sigma^{-1}$ a $\sigma^{-1}$-derivation on $R^{\text{op}}$ is an exercise in \cite{GW04} that can be solved without any use of associativity. Now, since $\alpha$ commutes with $\delta$ and $\sigma$, for any $r\in R^{\text{op}}$, $\sigma(\alpha(\sigma^{-1}(r)))=\alpha(\sigma(\sigma^{-1}(r)))=\alpha(r)$, so by applying $\sigma^{-1}$ to both sides, $\alpha(\sigma^{-1}(r))=\sigma^{-1}(\alpha(r))$. From this, it follows that $-\delta(\sigma^{-1}(\alpha(r)))=-\delta(\alpha(\sigma^{-1}(r)))=\alpha(-\delta(\sigma^{-1}(r)))$, which proves the first and second statement.

For the third statement, let us start by putting $S:=R^{\text{op}}[X;\sigma^{-1},-\delta\circ\sigma^{-1}]$ and $S':=R[X;\sigma,\delta]^\text{op}$, and then define a map $f\colon S\to S'$ by $\sum_{i=0}^nr_iX^i\mapsto \sum_{i=0}^nr_i\op X^i$ for $n\in\mathbb{N}$. We claim that $f$ is an isomorphism of hom-associative rings. First, note that an arbitrary element of $S'$ by definition is of the form $p:=\sum_{i=0}^ma_iX^i$ for some $m\in\mathbb{N}$ and $a_i\in R^\text{op}$. Then,
\allowdisplaybreaks
\begin{align*}
&p=\underbrace{X^m\cdot\sigma^{-m}(a_m)+b_{m-1}X^{m-1}+\dots +b_0}_{=a_mX^m}+\dots+\underbrace{X\cdot\sigma^{-1}(a_1)+\delta(\sigma^{-1}(a_1))}_{=a_1X}+a_0\\
&=X^m\cdot\sigma^{-m}(a_m)+X^{m-1}\cdot a'_{m-1}+\dots+X\cdot a_1'+a_0'\\
&=\sigma^{-m}(a_m)\op X^m+ a'_{m-1}\op X^{m-1}+\dots+\cdot a_1'\op X+a_0'\in\im f,
\end{align*}
for some $a'_{m-1}, b_{m-1}, \dots,a'_0,b_0\in R^\text{op}$, so $f$ is surjective. The second last step also shows that $\sum_{i=0}^m RX^i\subseteq \sum_{i=0}^m X^iR$ as sets, and a similar calculation shows that $\sum_{i=0}^m X^iR\subseteq \sum_{i=0}^m RX^i$, so that as sets, $\sum_{i=0}^m RX^i = \sum_{i=0}^m X^iR$. Hence, if $\sum_{i=0}^m r_i\op X^i=\sum_{j=0}^{m'}r'_j\op X^j$ for some $r_i,r'_j\in R^\text{op}$ and $m,m'\in\mathbb{N}$, then $m=m'$ and so 
\allowdisplaybreaks
\begin{align}
0=&\sum_{i=0}^m(r_i-r_i')\op X^i=\sum_{i=0}^mX^i\cdot (r_i-r_i')=\sum_{i=0}^m\sum_{j\in\mathbb{N}}\pi_j^i(r_i-r_i')X^j\nonumber\\
=&\sum_{j=0}^m\sum_{i=0}^m\pi_j^i(r_i-r_i')X^j\implies0=\sum_{i=0}^m\pi_j^i(r_i-r_i')X^j\quad\text{for } 0\leq j\leq m,\label{eq:r_i-sum}
\end{align}
where the implication comes from comparing coefficients with the left-hand side, which is equal to zero. Let us prove by induction that $r_j=r'_j$ for $0\leq j\leq m$. Put $k=m-j$, where $m$ is fixed, and consider the statement $\mathsf{P}(k)\colon$ $r_{m-k}=r'_{m-k}$ for $0\leq k\leq m$. 

Base case $(\mathsf{P}(0))\colon k=0\iff j=m$, so using that $\sigma$ is an automorphism,
\begin{equation*}
0\stackrel{\eqref{eq:r_i-sum}}{=}\sum_{i=0}^m\pi_m^i(r_i-r_i')X^m=\sigma^m(r_m-r_m')X^m\implies 0=r_m-r_m'.
\end{equation*}
Induction step (For $0\leq k\leq m\colon (\mathsf{P}(k)\to \mathsf{P}(k+1)$)): By putting $j=m-(k+1)$ and then using the induction hypothesis,
\begin{equation*}
0\stackrel{\eqref{eq:r_i-sum}}{=}\sum_{i=0}^m\pi_{m-(k+1)}^i(r_i-r_i')X^{m-(k+1)}=\sigma^{m-(k+1)}(r_{m-(k+1)}-r'_{m-(k+1)}),
\end{equation*}
which implies $0=r_{m-(k+1)}=r'_{m-(k+1)}$. Hence $r_j=r_j'$ for $0\leq j\leq m$, so $\sum_{i=0}^mr_i\op X^i=\sum_{j=0}^{m'}r_j'\op X^j\implies\sum_{i=0}^m r_iX^i=\sum_{j=0}^{m'}r_j'X^j$, proving that $f$ is injective. Additivity of $f$ follows immediately from the definition by using distributivity. Using additivity also makes it sufficient to consider only two arbitrary monomials $aX^m$ and $bX^n$ in $S$ when proving that $f$ is multiplicative. To this end, let us use the following notation for multiplication in $S$: $aX^m\bullet bX^n:=\sum_{i\in\mathbb{N}}\left(a\op\bar\pi_i^m(b)\right)X^{i+n}$, and then use induction over $n$ and $m$;

Base case $(\mathsf{P}(0,0))\colon f(a\bullet b)=f(a\op b)=a\op b=f(a)\op f(b)$. 

Induction step over $n$ ($\forall (m,n)\in\mathbb{N}\times\mathbb{N}\ (\mathsf{P}(m,n)\to\mathsf{P}(m,n+1))$): We know that $X\in N(S')$ by \autoref{prop:X-nucleus}, and so 
\begin{align*}
&f\left(aX^m\bullet bX^{n+1}\right)=f\left(\sum_{i\in\mathbb{N}}\left(a\op \bar\pi_i^m(b)\right)X^{i+n+1}\right)=\sum_{i\in\mathbb{N}} \left(a\op \bar\pi_i^m(b)\right)\op X^{i+n+1}\\
&=\left(\sum_{i\in\mathbb{N}} \left(a\op \bar\pi_i^m(b)\right)\op X^{i+n}\right)\op X=f(aX^m\bullet bX^n)\op X\\
&=\left(f(aX^m)\op f(bX^n)\right)\op X=f(aX^m)\op \left(f(bX^n)\op X\right)\\
&=f(aX^m)\op\left((b\op X^n)\op X\right)=f(aX^m)\op\left(b\op (X^n\op X)\right)\\
&=f(aX^m)\op\left(b\op X^{n+1}\right)=f(aX^m)\op f(bX^{n+1}).
\end{align*}

Induction step over $m$ ($\forall (m,n)\in\mathbb{N}\times\mathbb{N}\ (\mathsf{P}(m,n)\to\mathsf{P}(m+1,n))$): We know that $X\in N(S'^{\text{op}})\cap N(S)$ by \autoref{prop:X-nucleus}, and so by a straightforward calculation, $f\left(aX^{m+1}\bullet bX^n\right)=f\left(aX^{m+1}\right)\op f\left(bX^n\right)$.
Now, according to \autoref{def:morphism} with $R[X;\sigma,\delta]$ considered as a hom-associative algebra over the integers, we are done if we can prove that $f\circ\alpha=\alpha\circ f$ for the homogeneously extended map $\alpha$. Since both $\alpha$ and $f$ are additive, it again suffices to prove that $f\left(\left(\alpha(aX^m\right)\right)=\alpha\left(f\left(aX^m\right)\right)$ for some arbitrary monomial $aX^m$ in $R[X;\sigma,\delta]$. This is verified by a simple computation. 
\end{proof}

\begin{theorem}[Hilbert's basis theorem for hom-associative rings]\label{thm:hom-hilbert} Let $R$ be a unital, hom-associative ring with twisting map $\alpha$, $\sigma$ an automorphism and $\delta$ a $\sigma$-derivation that both commute with $\alpha$. Extend $\alpha$ homogeneously to $R[X;\sigma,\delta]$.  If $R$ is right (left) Noetherian, then so is $R[X;\sigma,\delta]$.
\end{theorem}

\begin{proof}This proof is an adaptation of a proof in \cite{GW04} to the hom-associative setting. Let us begin with the right case, and therefore assume that $R$ is right Noetherian. We wish to show that any right ideal of $R[X;\sigma,\delta]$ is finitely generated. Since the zero ideal is finitely generated, it is sufficient to show that any nonzero right ideal $I$ of $R[X;\sigma,\delta]$ is finitely generated. Let $J:=\{r\in R\colon rX^d+ r_{d-1}X^{d-1}+\dots + r_1X+r_0\in I, r_{d-1},\dots,r_0\in R\}$, i.e. $J$ consists of the zero element and all leading coefficients of polynomials in $I$. We claim that $J$ is a right ideal of $R$. First, one readily verifies that $J$ is an additive subgroup of $R$. Now, let $r\in J$ and $a\in R$ be arbitrary. Then there is some polynomial $p=rX^d+[\text{lower order terms}]$ in $I$. Moreover, $p\cdot\sigma^{-d}(a)=rX^d\cdot \sigma^{-d}(a)+[\text{lower order terms}]=\left(r\cdot\sigma^d(\sigma^{-d}(a))\right)X^d+[\text{lower order terms}]=(r\cdot a)X^d+[\text{lower order terms}]$, which is an element of $I$ since $p$ is. Therefore, $r\cdot a\in J$, so $J$ is a right ideal of $R$. 

Since $R$ is right Noetherian and $J$ is a right ideal of $R$, $J$ is finitely generated, say by $\{r_1,\dots,r_k\}\subseteq J$. All the elements $r_1,\dots,r_k$ are assumed to be nonzero, and moreover, each of them is a leading coefficient of some polynomial $p_i\in I$ of degree $n_i$. Put $n=\max(n_1,\dots,n_k)$. Then each $r_i$ is the leading coefficient of $p_i \cdot X^{n-n_i}=r_i X^{n_i}\cdot X^{n-n_i}+[\text{lower order terms}]=r_iX^n+[\text{lower order terms}]$, which is an element of $I$ of degree $n$.

Let $N:=\sum_{i=0}^{n-1} RX^i$. Then calculations similar to those in the proof of the third statement of \autoref{lem:opposite-hom-noetherian} show that as sets, $N=\sum_{i=0}^{n-1} RX^i = \sum_{i=0}^{n-1} X^iR$. By \autoref{prop:ore-module}, $N$ is then a hom-Noetherian right $R$-hom-module. Now, since $I$ is a right ideal of the ring $R[X;\sigma,\delta]$ which contains $R$, in particular, it is also a right $R$-hom-module. By \autoref{prop:hom-submodule-intersection}, $I\cap N$ is then a hom-submodule of $N$, and since $N$ is a hom-Noetherian right $R$-hom-module, $I\cap N$ is finitely generated, say by the set $\{q_1,q_2,\dots, q_t\}$.

Let $I_0$ be the right ideal of $R[X;\sigma,\delta]$ generated by 
\begin{equation*}
\left\{p_1\cdot X^{n-n_1},p_2\cdot X^{n-n_2},\dots,p_k\cdot X^{n-n_k},q_1,q_2,\dots,q_t\right\}.
\end{equation*}
Since all the elements in this set belong to $I$, we have that $I_0\subseteq I$. We claim that $I\subseteq I_0$. In order to prove this, pick any element $p'\in I$.

Base case ($\mathsf{P}(n)$): If $\deg p'<n$, $p'\in N=\sum_{i=0}^{n-1} RX^i$, so $p'\in I\cap N$. On the other hand, the generating set of $I\cap N$ is a subset of the generating set of $I_0$, so $I\cap N\subseteq I_0$, and therefore $p'\in I_0$. 

Induction step ($\forall m\geq n\ (\mathsf{P}(m)\rightarrow \mathsf{P}(m+1))$): Assume $\deg p'= m\geq n$ and that $I_0$ contains all elements of $I$ with $\deg < m$. Does $I_0$ contain all elements of $I$ with $\deg<m+1$ as well? Let $r'$ be the leading coefficient of $p'$, so that we have $p'=r'X^m+[\text{lower order terms}]$. Since $p'\in I$ by assumption, $r'\in J$. We then claim that $r'=\sum_{i=1}^k\sum_{j=1}^{k'}(\cdots((r_i\cdot a_{ij1})\cdot a_{ij2})\cdot\cdots)\cdot a_{ijk''}$ for some $k', k''\in\mathbb{N}_{>0}$ and some $a_{ij1}, a_{ij2},\ldots, a_{ijk''}\in R$. First, we note that since $J$ is generated by $\{r_1,r_2,\dots,r_k\}$, it is necessary that $J$ contains all elements of that form. Secondly, we see that subtracting any two such elements or multiplying any such element from the right with one from $R$ again yields such an element, and hence the set of all elements of this form is not only a right ideal containing $\{r_1,r_2,\dots,r_k\}$, but also the smallest such to do so. 

Recalling that $p_i\cdot X^{n-n_i}=r_iX^n+[\text{lower order terms}]$, $\left(p_i\cdot X^{n-n_i}\right)\cdot \sigma^{-n}(a_{ij1})=(r_i\cdot a_{ij1})X^n+[\text{lower order terms}]$, and by iterating this multiplication from the right, we set $c_{ij}:=\left(\cdots\left(\left(\left(p_i\cdot X^{n-n_i}\right)\cdot\sigma^{-n}(a_{ij1})\right)\cdot\sigma^{-n}(a_{ij2})\right)\cdot\cdots\right)\cdot\sigma^{-n}(a_{ijk''})$. Since $p_i\cdot X^{n-n_i}$ is a generator of $I_0$, $c_{ij}$ is an element of $I_0$ as well, and therefore also $q:=\sum_{i=1}^k\sum_{j=1}^{k'} c_{ij}\cdot X^{m-n}=r'X^m + [\text{lower order terms}].$ However, as $I_0\subseteq I$, we also have that $q\in I$, and since $p'\in I$, $(p'-q)\in I$. Now, $p'=r'X^m+[\text{lower order terms}]$, so $\deg(p'-q)<m$, and therefore $(p'-q)\in I_0$. This shows that $p'=(p'-q)+q$ is an element of $I_0$ as well, and thus $I=I_0$, which is finitely generated.

For the left case, first note that any hom-associative ring $S$ is right (left) Noetherian if and only if $S^\text{op}$ is left (right) Noetherian, due to the fact that any right (left) ideal of $S$ is a left (right) ideal of $S^\text{op}$, and vice versa. Now, assume that $R$ is left Noetherian. Then, $R^\text{op}$ is right Noetherian, and using \ref{le:Rop-automorphism} and \ref{le:Rop-derivation} in \autoref{lem:opposite-hom-noetherian}, $\sigma^{-1}$ is an automorphism and $-\delta\circ\sigma^{-1}$ a $\sigma^{-1}$-derivation on $R^\text{op}$ that both commute with $\alpha$. Hence, by the previously proved right case, $R^{\text{op}}[X;\sigma^{-1},-\delta\circ\sigma^{-1}]$ is right Noetherian. By \ref{le:Rop-equality} in \autoref{lem:opposite-hom-noetherian}, $R^{\text{op}}[X;\sigma^{-1},-\delta\circ\sigma^{-1}]\cong R[X;\sigma,\delta]^\text{op}$. One verifies that surjective morphisms between hom-associative rings preserve the Noetherian conditions \ref{eq:hom-noetherian-ring-1}, \ref{eq:hom-noetherian-ring-2}, and \ref{eq:hom-noetherian-ring-3} in \autoref{prop:hom-noetherian-conditions-for-rings} by examining the proof of \autoref{prop:surjective-hom-noetherian-morphism}, changing the module morphism to that between rings instead, and ``submodule'' to ``ideal''. Therefore, $R[X;\sigma,\delta]^\text{op}$ is right Noetherian, so $R[X;\sigma,\delta]$ is left Noetherian.
\end{proof}

\begin{remark}\label{re:assoc-hilbert} By putting $\alpha=\mathrm{id}_R$ in \autoref{thm:hom-hilbert}, we recover the classical Hilbert's basis theorem for Ore extensions. 
\end{remark}

\begin{corollary}[Hilbert's basis theorem for non-associative rings]\label{cor:non-hilbert}
Let $R$ be a unital, non-associative ring, $\sigma$ an automorphism and $\delta$ a $\sigma$-derivation on $R$. If $R$ is right (left) Noetherian, then so is $R[X;\sigma,\delta]$.
\end{corollary}

\begin{proof} Put $\alpha\equiv0$ in \autoref{thm:hom-hilbert}.
\end{proof}

\section{Examples}\label{sec:examples}
Here we provide some examples of unital, non-associative and hom-associative Ore extensions which are all Noetherian by the above theorem. First, recall that there are, up to isomorphism, only four normed, unital division algebras over the real numbers: the real numbers themselves, the complex numbers, the quaternions ($\mathbb{H}$), and the octonions ($\mathbb{O}$)~\cite{UW60}. The largest of the four are the octonions, and while sharing the property of not being commutative with the quaternions, the octonions are the only ones that are not associative. All of the four algebras above are Noetherian, and hence also all iterated Ore extensions of them: let $D$ be any unital division algebra, and $I$ any nonzero right ideal of $D$. If $a\in D$ is an arbitrary nonzero element, then  $1=a\cdot a^{-1}\in I$, so $I=D$, and analogously for the left case. As an ideal of itself, $D$ is finitely generated (by $1$, for instance), as is the zero ideal. 

The derivations on any normed division algebra $D$ is a linear combination of derivations $\delta_{a,b}$ where $a,b\in D$, defined by $\delta_{a,b}(x):=[[a,b],x]-3(a,b,x)$ for all $x\in D$~\cite{Sch66}. These derivations are called \emph{inner}, and in particular, all derivations on $\mathbb{H}$ are of the form $[a,\cdot]$ for some $a\in\mathbb{H}$.

Given a unital and associative algebra $A$ with product $\cdot$ over a field of characteristic different from two, one may define a unital and non-associative algebra $A^+$ by using the \emph{Jordan product} $\{\cdot,\cdot\}\colon A^+\to A^+$. This is given by $\{a,b\}:=\frac{1}{2}\left(a\cdot b + b\cdot a\right)$ for any $a,b\in A$. $A^+$ is then a \emph{Jordan algebra}, i.e. a commutative algebra where any two elements $a$ and $b$ satisfy the \emph{Jordan identity}, $\left\{\{a,b\}\{a,a\}\right\}=\left\{a,\{b,\{a,a\}\}\right\}$. Since inverses on $A$ extend to inverses on $A^+$, one may infer that if $A=\mathbb{H}$, then $A^+$ is also Noetherian. Using the standard notation $i,j,k$ for the \emph{quaternion units} in $\mathbb{H}$ with defining relation $i^2=j^2=k^2=ijk=-1$, one can see that $\mathbb{H}^+$ is not associative as e.g. $(i,i,j)_{\mathbb{H}^+}:=\{\{i,i\},j\}-\{i,\{i,j\}\}=-j$.

\begin{example}Let $\sigma$ be the automorphism on $\mathbb{H}$ defined by $\sigma(i)=-i$, $\sigma(j)=k$, and $\sigma(k)=j$. Any automorphism on $\mathbb{H}$ is also an automorphism on $\mathbb{H}^+$, and hence $\mathbb{H}^+\left[X;\sigma,0_\mathbb{H}\right]$ is a unital, non-associative Ore extension. $\mathbb{H}^+\left[X;\sigma,0_\mathbb{H}\right]$ is then Noetherian by \autoref{cor:non-hilbert}.
\end{example}

\begin{example}\label{ex:inner-der-plus}Let $[j,\cdot]_\mathbb{H}$ be the inner derivation on $\mathbb{H}$ induced by $j$. Any derivation on $\mathbb{H}$ is also a derivation on $\mathbb{H}^+$, and so we may form the unital, non-associative Ore extension $\mathbb{H}^+\left[X;\mathrm{id}_\mathbb{H},[j,\cdot]_\mathbb{H}\right]$ which is Noetherian due to \autoref{cor:non-hilbert}. 
\end{example}

\begin{example}\label{ex:inner-der}From the Jordan identity one may infer that a map $\delta_{a,b}\colon J\to J$ defined by $\delta_{a,b}(x):=(a,x,b)_{J}$ for any $a,b,x\in J$ where $J$ is a Jordan algebra, is a derivation, called an inner derivation. On $\mathbb{H}^+$ one could for instance take $a=i$ and $b=j$, resulting in $\delta_{i,j}(x)=\{\{i,x\},j\}-\{i,\{x,j\}\}$ for any $x\in\mathbb{H}^+$. Then $\mathbb{H}^+\left[X;\mathrm{id}_{\mathbb{H}},\delta_{i,j}\right]$ is a unital, non-associative Ore extension which is Noetherian by \autoref{cor:non-hilbert}. 
\end{example}

\begin{example}
Take any derivation on $\mathbb{O}$, e.g. $\delta_{i,j}$ defined by $\delta_{i,j}(x):=[[i,j],x]-3(i,j,x)$ for any $x\in\mathbb{O}$. Then $\mathbb{O}[X;\mathrm{id}_{\mathbb{O}},\delta_{i,j}]$ is a unital, non-associative Ore extension which is Noetherian by \autoref{cor:non-hilbert}.
\end{example}

\begin{example}\label{ex:weyl-octonions}One may define an \emph{octonionic Weyl algebra} $A(\mathbb{O})$ as the tensor product of the usual Weyl algebra $A(\mathbb{R})$ over $\mathbb{R}$, and $\mathbb{O}$. $A(\mathbb{O})$ is then a free module of finite rank over $A(\mathbb{R})$, and hence it is Noetherian. One may also define an octonionic Weyl algebra as an iterated Ore extension of $\mathbb{O}$. Using this latter approach, let us first mention that for any unital, non-associative ring $R$, the \emph{non-associative Weyl algebra over $R$} was introduced in \cite{NOR15} as the iterated, unital, non-associative Ore extension $R[Y][X;\mathrm{id}_R,\delta]$ where $\delta\colon R[Y]\to R[Y]$ is an $R$-linear map such that $\delta(1)=0$. The unital, non-associative Ore extension of $\mathbb{O}$ in the indeterminate $Y$ is the unital and non-associative polynomial ring $\mathbb{O}[Y;\mathrm{id}_\mathbb{O},0_{\mathbb{O}}]$, for which we write $\mathbb{O}[Y]$. Let $\delta\colon \mathbb{O}[Y]\to \mathbb{O}[Y]$ be the $\mathbb{O}$-linear map defined on monomials by $\delta\left(aY^m\right)=maY^{m-1}$ for arbitrary $a\in\mathbb{O}$ and $m\in\mathbb{N}$, with the interpretation that $0aY^{-1}$ is $0$. One readily verifies that $\delta$ is an $\mathbb{O}$-linear derivation on $\mathbb{O}[Y]$, and by \autoref{re:sigma-derivation}, $\delta(1)=0$. We define an octonionic Weyl algebra $\mathbb{O}[Y][X;\mathrm{id}_{\mathbb{O}[Y]},\delta]$, where $\delta$ is the aforementioned derivation. By using \autoref{cor:non-hilbert} twice, $\mathbb{O}[Y][X;\mathrm{id}_{\mathbb{O}[Y]},\delta]$ is Noetherian. Moreover, $\mathbb{O}[Y][X;\mathrm{id}_{\mathbb{O}[Y]},\delta]\cong A(\mathbb{O})$.
\end{example}

\begin{example}\label{ex:q-weyl-octonions}
For any $q\in\mathbb{R}\backslash\{0,1\}$, one may define an \emph{octonionic $q$-Weyl algebra} $A_q(\mathbb{O})$ as the tensor product of the usual $q$-Weyl algebra $A_q(\mathbb{R})$ over $\mathbb{R}$, and $\mathbb{O}$. In particular, $A_q(\mathbb{O})$ is a free module of finite rank over $A_q(\mathbb{R})$, and hence it is Noetherian. Alternatively, one may see that $A_q(\mathbb{O})$ is Noetherian by noting that it is the iterated Ore extension $\mathbb{O}[Y][X;\sigma,\delta]$ where $\sigma$ is the automorphism on $\mathbb{O}[Y]$ defined by $\sigma(Y)=qY$, and $\delta$ the \emph{Jackson $q$-derivative}, i.e. the $\sigma$-derivation defined on an arbitrary polynomial $p(Y)\in\mathbb{O}[Y]$ by 
\begin{equation*}
\delta(p(Y)):=\frac{p(qY)-p(Y)}{qY-Y}=\frac{\sigma(p(Y))-p(Y)}{\sigma(Y)-Y}.
\end{equation*}
\end{example}

\begin{example}\label{ex:unital}
This example is a slight generalization of Example 1.1 in \cite{FG09}. Let $R$ and $S$ be associative, commutative, and unital rings, and $f\colon R\to S$ a homomorphism. Further assume that $R$ is Noetherian. Let $A$ be a non-associative, non-unital, Noetherian $S$-algebra, and define a multiplication $\cdot$ on $U:=R\times A$ by $(r_1,a_1)\cdot (r_2,a_2):=(r_1r_2,f(r_1)a_2+f(r_2)a_1+a_1a_2)$ for any $r_1,r_2\in R$ and $a_1,a_2\in A$. $U$ is then unital with identity element $(1,0)$, and by defining a twisting map $\alpha$ on $U$ by $\alpha(r,a):=(pr,0)$ for any $r\in R$, $a\in A$, and $p\in\ker f$, $U$ is hom-associative. Moreover, $U$ is Noetherian, and if $A$ is not associative, then $U$ is not associative. Now, let $\sigma_A$ be an automorphism on $A$. Then $\sigma$ defined by $\sigma(r,a):=(r,\sigma_A(a))$ is an automorphism on $U$. Moreover, if $\delta_A$ is a $\sigma_A$-derivation on $A$, then $\delta$ defined by $\delta(r,a):=(0,\delta_A(a))$ is a $\sigma$-derivation on $U$, and both $\delta$ and $\sigma$ commute with $\alpha$. Hence, by \autoref{thm:hom-hilbert}, $U[X;\sigma,\delta]$ is Noetherian. Here, one could e.g. take  $R=\mathbb{R}[Y]$, $S=\mathbb{R}$, $f\colon\mathbb{R}[Y]\to\mathbb{R}$ the evaluation homomorphism at zero, $p\in\mathbb{R}[Y]$ any polynomial without a constant term, and $A$, $\sigma_A$, and $\delta_A$ any $\mathbb{R}$-algebra, $\sigma$, and $\delta$, respectively, from the previous examples.  

We here include a proof that $U$ is Noetherian. Suppose we have an ascending chain of right (left) ideals, $I_1 \subseteq I_2 \subseteq \ldots$, in $U$. Define $J_j = \{r \in R \, | \,  \exists a \in A: (r,a) \in I_j\}.$ This is an ideal in $R$. Also define $H_j = \{a \in A \, | \,  (0,a) \in I_j\}.$ This is a right (left) ideal in $A$. We thus have two ascending chains,  $J_1\subseteq J_2\subseteq \ldots$ and $H_1\subseteq H_2\subseteq \ldots$, in $R$ and $A$, respectively. Since $R$ and $A$ are Noetherian there is some integer $n$ such that if $k>n$ then $J_k= J_n$ and $H_k = H_n$. We claim that in fact also $I_k = I_n $. Let $(r,a) \in I_k$. Then $r \in J_k = J_n$ so there is $a' \in A$ such that $(r,a') \in I_n$. It follows that $a-a' \in H_k = H_n$, which implies $(0,a-a') \in I_n$. Hence $(r,a) = (r,a')+(0,a-a')$ is a sum of two elements in $I_n$ and therefore belongs to $I_n$.
\end{example}

\begin{example}\label{ex:lars}
Let $R$ be a unital, non-associative, Noetherian ring, and denote by $I$ the ideal of $R$ generated by all expressions of the form $r(st)-(rs)t$ where $r,s,t\in R$. Define $S:=R/I$ and let $\pi\colon R\to S$ be the natural homomorphism. Set $U:=R\times S$ and define a multiplication $\cdot$ on $U$ by $(r_1,s_1)\cdot (r_2,s_2):=(r_1r_2,\pi(r_1)s_2+s_1\pi(r_2)+s_1s_2)$ for all $r_1,r_2\in R$ and $s_1,s_2\in R$. $U$ is unital with identity element $(1,0)$, and the map $\alpha$ defined by $\alpha(r,s):=(0,\pi(r)+s)$ for all $r\in R$ and $s\in S$ is a well-defined twisting map that makes $U$ hom-associative. Since $R$ is Noetherian, so is $S$, and by the same argument as in \autoref{ex:unital}, $U$ is Noetherian. Moreover, if $R$ is not associative, then $U$ is not associative. Now, let $\sigma_R$ be an endomorphism on $R$. Then $\sigma_R(I)\subseteq I$, which guarantees  the naturally extended endomorphism $\sigma_S$ on $S$ to be well-defined. By defining $\sigma(r,s):=(\sigma_R(r),\sigma_S(s))$, we get an endomorphism $\sigma$ on $U$. Similarly, any $\sigma_R$-derivation $\delta_R$ satisfies $\delta_R(I)\subseteq I$, and hence the naturally extended $\sigma_S$-derivation $\delta_S$ is well-defined, and in turn gives rise to a $\sigma$-derivation $\delta$ on $U$ defined by $\delta(r,s):=(\delta_R(r),\delta_S(s))$. Now, assume that $\sigma_R$ is an automorphism. Then it is clear that $\sigma_S$ is surjective. Moreover, $\sigma_R^{-1}(I)\subseteq I$, which in turn implies that $\sigma_S$ is injective. Hence $\sigma$ is an automorphism. Moreover, $\alpha$ commutes with both $\delta$ and $\sigma$, and therefore $U[X;\sigma,\delta]$ is Noetherian. Here, one could e.g. take $R$ to be any base ring together with $\sigma$ and $\delta$ from the previous examples. 
\end{example}

\section{Acknowledgments}
We wish to thank Patrik Lundstr{\"o}m for pointing out that the associator identity could be used in proving \autoref{prop:X-nucleus} in a simpler way than was originally done. We would also like to thank Lars Hellstr\"{o}m for providing the unital hom-ring in \autoref{ex:lars}, and the referee for giving concrete suggestions on how to improve the article.


\begin{thebibliography}{99}

\bibitem{Bac19}
P.~B{\"a}ck,
\emph{Notes on formal deformations of quantum planes and universal enveloping algebras},
J. Phys.: Conf. Ser. \textbf{1194}(1) (2019).

\bibitem{BR19}
P.~B{\"a}ck, J.~Richter,
\emph{On the hom-associative Weyl algebras}, J. Pure Appl. Algebra {\bf 224}(9) (2020).

\bibitem{BRS18}
P.~B{\"a}ck, J.~Richter, S.~Silvestrov,	
\emph{Hom-associative Ore extensions and weak unitalizations},
Int. Electron. J. Algebra \textbf{24} (2018), pp. 174--194.

\bibitem{FG09}
Y.~Fr{\'e}gier, A.~Gohr, \emph{On unitality conditions for hom-associative algebras}, \texttt{arXiv:0904.4874}.

\bibitem{GW04}
K.~R.~Goodearl, R.~B.~Warfield,
\emph{An Introduction to Noncommutative Noetherian Rings},
Cambridge University Press, Cambridge U.~K., 2004.

\bibitem{HLS03} 
J.~T.~Hartwig, D.~Larsson, S.~D.~Silvestrov,
\emph{Deformations of Lie algebras using $\sigma$-derivations},
J. Algebra {\bf 295} (2006), pp. 314--361.

\bibitem{MS06}
A.~Makhlouf, S.~D.~Silvestrov,
\emph{Hom-algebra structures},
 J. Gen. Lie Theory Appl. \textbf{2} (2008), pp. 51--64.
 
 \bibitem{MS10b}
 A.~Makhlouf, S.~Silvestrov,
 \emph{Hom-algebras and Hom-coalgebras},
 J. Algebra Appl., \textbf{9} (2010), pp. 553--589.
 
 \bibitem{MS09}
 A.~Makhlouf, S.~Silvestrov,
 \emph{Hom-Lie Admissible Hom-Coalgebras and Hom-Hopf Algebras},
 Generalized Lie Theory in Mathematics, Physics and Beyond, eds. S.~Silvestrov, E.~Paal, V.~Abramov, A.~Stolin. 
 Springer, Berlin, Heidelberg, 2009.
 
 \bibitem{MS10a}
 A.~Makhlouf, S.~Silvestrov,
\emph{Notes on 1-parameter formal deformations of Hom-associative and Hom-Lie algebras},
 Forum Math. \textbf{22} (2010), pp. 715--739.

\bibitem{Nys13}
P.~Nystedt,
\emph{A combinatorial proof of associativity of Ore extensions},
Discrete Math. {\bf 313} (2013), pp. 2748--2750.

\bibitem{NOR15} 
P.~Nystedt, J.~{\"O}inert, J.~Richter,
\emph{Non-associative Ore extensions},
Isr. J. Math. {\bf 224} (2018), pp. 263--292.

\bibitem{NOR17}
P.~Nystedt, J.~{\"O}inert, J.~Richter,
\emph{Simplicity of Ore monoid rings},
J. Algebra {\bf 530} (2019), pp. 69--85.

\bibitem{Ore33}
O.~Ore,
\emph{Theory of Non-Commutative Polynomials},
Ann. of Math. \textbf{34} (1933), pp. 480--508.

\bibitem{Sch66}
R.~D.~Schafer,
\emph{An Introduction to Nonassociative Algebras},
Academic P., New York, 1966.

\bibitem{UW60}
K.~Urbanik, F.~B.~Wright,
\emph{Absolute-valued algebras},
Proc. Amer. Math. Soc. \textbf{11} (1960), pp. 861--866.

\end{thebibliography}
\end{document}